%% file: Laver_Generalised-preprint.tex
\documentclass{article}
 
\usepackage{amsthm}
\usepackage{amsmath}
\usepackage{amssymb}
\usepackage{amsfonts}
\usepackage{amsxtra} 
\usepackage{mathrsfs}
\usepackage{color}
\usepackage{graphicx}
\usepackage{enumerate} 
\usepackage[all]{xy}
\usepackage{authblk}

\title{Laver Trees in the Generalized Baire Space}

\author{Yurii Khomskii\footnote{Amsterdam University College, Postbus 94160, 1090 GD Amsterdam, The Netherlands and Universi\"at Hamburg, Bundesstra{\ss}e 55, 20146 Hamburg, Germany.  \texttt{yurii@deds.nl}. Supported by the European Union’s Horizon 2020 research and innovation programme under the Marie Sk\l odowska-Curie grant agreement No 706219 (REGPROP)} , $\;$ Marlene Koelbing\footnote{Kurt G\"odel Research Center (KGRC), Universit\"at Wien, Oskar-Morgenstern-Platz 1, 1090 Vienna, Austria. \texttt{marlenekoelbing@web.de}.Supported by the \"OAW Doc fellowship.}, $\;$ \\  Giorgio Laguzzi\footnote{Albert-Ludwigs-Universit\"at Freiburg, Ernst-Zermelo Str. 1, 79104 Freiburg im Breisgau, Germany.
 \texttt{giorgio.laguzzi@libero.it } }, $\;$ Wolfgang Wohofsky\footnote{Kurt G\"odel Research Center (KGRC), Universit\"at Wien, Oskar-Morgenstern-Platz 1, 1090 Vienna, Austria.
 \texttt{wolfgang.wohofsky@gmx.at} Supported by the German Research Foundation (DFG) under Grant SP 683/4-1, and the Austrian Science Fund (FWF) under International Project number: I 4039 \medskip \\ This project was partially supported by the Isaac Newton Institute for Mathematical Sciences in the programme Mathematical, Foundational and Computational Aspects of the Higher Infinite (HIF) funded by EPSRC grant EP/K032208/1.}
}

\include{newcommands}

\renewcommand{\kk}{\kappa^\kappa}
\newcommand{\klk}{\kappa^{{<}\kappa}}
\newcommand{\dk}{2^{\kappa}}
\newcommand{\dlk}{2^{{<}\kappa}}
\newcommand{\D}{{\mathcal{D}}}
\newcommand{\len}{{\rm len}}

\renewcommand{\T}{\mathfrak{T}}
\newcommand{\Txp}{\T_{\dot{x},p}}

\newcommand{\Tdq}{\T_{\dot{d},q}}

\newcommand{\xxg}{\dot{x}_{gen}}
\newcommand{\xg}{x_{gen}}
\renewcommand{\d}{{\dot{d}}}

\newcommand{\q}[1]{\ulcorner #1 \urcorner }

\renewcommand{\thru}{{\uparrow}}
\newcommand{\PIL}{\mathbb{PL}}
\renewcommand{\GCH}{{\sf GCH}}
\renewcommand{\CH}{{\sf CH}}

\begin{document}

\maketitle

\begin{abstract} We prove that any suitable generalization of Laver forcing  to the space $\kappa^\kappa$, for uncountable regular $\kappa$, necessarily adds a  Cohen $\kappa$-real. We also study a  dichotomy  and an ideal naturally related to generalized Laver forcing. Using this dichotomy, we prove the following stronger result: if $\klk=\kappa$, then every ${<}\kappa$-distributive tree forcing on $\kk$ adding a dominating $\kappa$-real which is the image of the generic under a continuous function in the ground model, adds a Cohen $\kappa$-real.  This is a contribution to the study of generalized Baire spaces and answers a question from  \cite{MontoyaEtAl}.
\end{abstract}


\section{Introduction} \label{sectionIntroduction}

In set theory of the reals, a basic question  is whether a forcing adds  \emph{Cohen reals} or  \emph{dominating reals}. It is well-known that {Cohen forcing} adds Cohen but not dominating reals while {Laver forcing} does the opposite. In the language of {cardinal characteristics of the continuum}, this means that an appropriate iteration of Cohen forcing starting from  $\CH$ yields  a model where $ \bb < \cov(\M)$, while  an appropriate iteration of Laver forcing starting from   $\CH$ yields a model where  $ \cov(\M) <   \bb$. 

\bigskip
In recent years, the study of \emph{generalized Baire spaces} has caught the attention of an increasing number of  set theorists. For a regular, uncountable cardinal $\kappa$ one considers  elements of  $\kk$ or $\dk$ as ``$\kappa$-reals'' and looks at the corresponding space with the bounded topology. The \emph{generalized Cantor space} is defined analogously using $\dk$ and $\dlk$.

It is straightforward to generalize the above notions from the classical to the generalized Baire spaces. Thus, we have the concepts \emph{dominating $\kappa$-real} and the cardinal characteristic $\bb_\kappa$ (see Definition \ref{dom}). Likewise, we can define $\M_\kappa$ as the ideal of \emph{$\kappa$-meager} sets, i.e., those obtained by $\kappa$-unions of nowhere dense, giving rise to the cardinal characteristic $\cov(\M_\kappa)$ defined in the usual way. \emph{$\kappa$-Cohen forcing} is the partial order of  basic open sets ordered by inclusion. 

It is not hard to see that $\kappa$-Cohen forcing does not add dominating $\kappa$-reals, so an appropriate iteration of  $\kappa$-Cohen forcing, starting from a model of $\GCH$, yields a model in which $ \bb_\kappa  < \cov(\M_\kappa)$, mirroring the classical situation. A natural method for the converse direction, i.e., proving the  consistency of $\cov(\M_\kappa) < \bb_\kappa$, would be to iterate a forcing which adds dominating $\kappa$-reals but not Cohen $\kappa$-reals. The authors of   \cite[p.\ 36]{MontoyaEtAl}  asked whether a forcing with such a property existed, and in particular, whether some generalization of Laver forcing had this property. 

\bigskip
In this paper,  we  show that \emph{any}  generalization of Laver forcing necessarily adds a Cohen $\kappa$-real (Theorem \ref{Mainthm}). If we assume $\klk=\kappa$, then this holds for an even wider class of trees (Theorem \ref{Mainthm2}). Later, we use a dichotomy result and similar techniques to show  that if $\klk = \kappa$ and $\IP$ is any ${<}\kappa$-distributive forcing whose conditions are limit-closed trees on $\klk$, and which adds a dominating $\kappa$-real obtained as the image of the generic under a continuous function in the ground model, then $\IP$ necessarily adds a Cohen $\kappa$-real (Theorem \ref{MainThm3}). It is an open question whether there exists some other ${<}\kappa$-distributive and/or  ${<}\kappa$-closed forcing which adds dominating $\kappa$-reals but not  Cohen $\kappa$-reals (Question \ref{qqq}).\footnote{In an earlier version of this paper, we claimed  that every ${<}\kappa$-closed forcing adding dominating $\kappa$-reals adds Cohen $\kappa$-reals, but the proof contained a gap, so, to our knowledge, the question is still open.}

We should note that a model for $\cov(\M_\kappa) < \bb_\kappa$ was recently constructed by Shelah (private communication). However, Shelah's method was to start from a model of  $\cov(\M_\kappa) = \bb_\kappa = \dk > \kappa^+$ and  add a witness to $\cov(\M_\kappa) = \kappa^+$ by a short forcing iteration. It is therefore still open whether an alternative proof exists by using a forcing iteration starting from a model of $\GCH$ which adds dominating $\kappa$-reals and no Cohen $\kappa$-reals.



\bigskip

When working in generalized Baire spaces, a common assumption is $\klk = \kappa$, which is sufficient to prove many pleasant properties of generalized Baire spaces, e.g., that the topology has a base of size $\kappa$. Nevertheless, our first main theorem (Theorem \ref{Mainthm}) is proved in generality and does not depend on this assumption, whereas the other main results (Theorem \ref{Mainthm2} and Theorem \ref{MainThm3}) do.

\bigskip The first main result is proved in Section \ref{game}. Motivated by the methods used there, in Section \ref{goldstern} we look at the ideal related to generalized Laver forcing and prove a somewhat surprising result concerning a generalization of the dichotomy for Laver forcing from \cite{GoldsternGame}. This dichotomy is used in  Section \ref{sectionbetter} to extend our first main result to arbitrary ${<}\kappa$-distributive tree forcings. 



\bigskip
\section{Preliminaries and definitions} \label{definitions}
 
We work in the setting where $\kappa$ is an uncountable, regular cardinal, and consider the \emph{generalized Baire space} $\kk$ with the bounded topology  generated by basic open sets of the form $[\sigma] := \{x \in \kk  :  \sigma \subseteq x\}$ for $\sigma \in \klk$. The  \emph{generalized Cantor space} $\dk$ is defined analogously. 

We refer the reader to \cite{FriedmanHyttinenKulikov} for a good introduction to generalized Baire spaces, and to  \cite{OpenQuestions} for an overview of the current state of the field and a list of open problems.



\begin{Def} \label{dom}  Let $f,g \in \kk$. We say that $g$ \emph{dominates} $f$, notation $f \leq^* g$, iff $\exists \alpha_0 \: \forall \alpha > \alpha_0 \: (f(\alpha) \leq g(\alpha))$. The generalized bounding number $\bb_\kappa$ is defined as the least size of a family $F \subseteq \kk$ such that for all $g \in \kk$ there is $f \in F$ such that $f \not\leq^* g$. If $M$ is a model of set theory, then $d$ is a \emph{dominating $\kappa$-real} over $M$ if $d$ dominates every $f \in \kk \cap  M$.
\end{Def}


A \emph{tree} in $\klk$  is a subset closed under initial segments. If $T$ is a tree, we use $[T]$ to denote the set of branches (of length $\kappa$) through $T$, that is $[T] := \{x \in \kk \: : \: \forall \alpha < \kappa \; (x \till \alpha \in T)\}$. The same holds for trees in $\dlk$. For $\sigma \in T$ we use the notation $T \thru \sigma := \{\tau \in T \: : \: \sigma \subseteq \tau \; \lor \; \tau \subseteq \sigma\}$.  A tree $T \subseteq \klk$ is called \emph{limit-closed}\footnote{Other terminology used  is ``${<}\kappa$-closed'' and ``sequentially closed''.} if for any limit ordinal $\lambda<\kappa$ and any $\subseteq$-increasing sequence $\left<\sigma_\alpha \: : \: \alpha < \lambda \right>$ from $T$, the limit of the sequence $\sigma := \bigcup_{\alpha<\lambda}\sigma_\alpha$ is itself an element of $T$. We call a set $C$ \emph{superclosed} if $C = [T]$ for a limit-closed tree $T$. 

Every closed subset of $\kk$ is the set of branches through a tree but not necessarily a limit-closed tree, so one could say that being superclosed is a topologically stronger property than being closed. We will also need to consider sets of branches   of length shorter than $\kappa$. For any limit ordinal $\lambda < \kappa$  we use the notation $[T]_\lambda := \{\sigma \in \kappa^\lambda : \forall \alpha<\lambda \; (\sigma \till \alpha \in T)\}$. Notice that $T$ is limit-closed iff $[T]_\lambda \subseteq T$ for all limit ordinals $\lambda < \kappa$. 

 
\begin{Def} A \emph{Laver tree} is a tree $T \subseteq \wlw$ with the property that  for every $\sigma \in T$ extending $\stem(T)$, $|\Succ_T(\sigma)| = \omega$. \emph{Laver forcing} $\IL$ is the partial order of Laver trees ordered by inclusion. \end{Def}

Laver forcing adds dominating reals while satisfying the so-called  \emph{Laver property}, a well-known iterable property implying that no Cohen reals are added.  There have been several attempts in the literature to generalize Laver forcing to $\kk$. 

\begin{Def} \label{lavdef} A \emph{$\kappa$-Laver tree} is a tree $T \subseteq \klk$ which is \emph{limit-closed} and such that for every $\sigma \in T$ extending $\stem(T)$, $|\Succ_T(\sigma)| = \kappa$. Let $\IL_\kappa$ denote the partial order of all $\kappa$-Laver trees ordered by inclusion. \end{Def}

This partial order itself is not well-suited as a forcing on $\kk$ and has never been proposed as an option.\footnote{It is not hard to see that such a partial order would not be ${<}\kappa$-closed, and in fact not even  $\omega$-distributive. Compare this to a recent result of Mildenberger and Shelah \cite{MiShe} showing that a similarly ``plain'' version of $\kappa$-Miller forcing collapses $2^\kappa$ to $\omega$.} But there have been other attempts at generalizations of Laver forcing, usually by putting stronger requirements on ``splitting'' in the tree. For example,  \emph{club Laver forcing}  (see \cite{GeneralizedVadim}) consists of trees satisfying the additional condition ``$\Succ_T(\sigma)$ contains a club on $\kappa$'' for all $\sigma$ beyond the stem. This forcing is ${<}\kappa$-closed and adds a dominating $\kappa$-real, but it is easy to see that it also adds a Cohen $\kappa$-real: if $S$ is a stationary, co-stationary subset of $\kappa$ and  $\varphi: \kk \to \dk$ is given by $\varphi(x)(\alpha) = 1 \; \Leftrightarrow   \; x(\alpha) \in S$, then  $\varphi(\xg)$ is a Cohen $\kappa$-real. 

Yet another attempt  is  \emph{measure-one Laver forcing}, where the requirement is strengthened to ``$\Succ_T(\sigma) \in U$'' for some ${<}\kappa$-complete ultrafilter on a measurable cardinal $\kappa$. This forcing is also ${<}\kappa$-closed and adds a dominating $\kappa$-real, and until now it was not known whether it adds a Cohen $\kappa$-real.  Of course, one could think of further clever requirements on Laver trees in order to ensure that no Cohen $\kappa$-reals are added.

However, by the results of this paper, none of these approaches can work.  

\newcommand{\Cof}{{\rm Cof}}
\newcommand{\cf}{{\rm cf}}

\bigskip
\section{The Supremum Game} \label{game}

In this section we will prove our first main result. The  main ingredient of our proofs in this and subsequent sections is the following game.

\begin{Def} Let $S \subseteq \kappa$. The \emph{supremum game} $G^{\rm sup}(S)$ is played by two players, for $\omega$ moves, as follows:

$$\begin{array}{r||cccccc}
{\rm I} & A_0 & & A_1 & & \dots \\ \hline
{\rm II} & & \beta_0 & & \beta_1 & & \dots \end{array}$$
where  $A_n \subseteq \kappa$,  $|A_n| = \kappa$ and   $\beta_n \in A_n$ for all $n<\omega$.   Player II wins   iff $\sup \{ \beta_n : n < \omega \} \in S$. \end{Def}

\begin{Lem} \label{gamelemma} Let $S$ be a stationary subset of $\Cof_\omega(\kappa) = \{\alpha <\kappa \: : \: \cf(\alpha) = \omega\}$. Then Player I does not have a winning strategy in $G^{\rm sup}(S)$. \end{Lem}

\begin{proof} Let $\sigma$ be a strategy for Player I in $G^{\rm sup}(S)$. We will show that $\sigma$ is not a winning strategy. Let $\theta$ be sufficiently large and let $M \prec \mathcal{H}_\theta$ be an elementary submodel such that $\sigma \in M$, $|M| < \kappa$, and $\delta:= \sup(M \cap \kappa) \in S$. Note that we can always do that, because the set $\{\sup(M \cap \kappa)  :  M \prec \mathcal{H}_\theta, \; \sigma \in M,\; |M| < \kappa\}$ contains a club.

\p Fix a  sequence $\left<\gamma_n \: : \: n < \omega\right>$ cofinal in $\delta$, such that every $\gamma_n \in M$ (but the sequence itself is not). Inductively, Player II will construct a run of the game according to strategy $\sigma$. 

\p At each step $n$, inductively assume  $A_k$ and $\beta_k$ for $k<n$ have been fixed according to the rules of the game and the strategy $\sigma$, and assume they are all in $M$. Let $A_n := \sigma(A_0, \beta_0, \dots, A_{n-1}, \beta_{n-1})$. Since the finite sequence was in $M$ and the strategy $\sigma$ is in $M$,  $A_n$ is also in $M$. Furthermore, since $|A_n| = \kappa$, the following statement is true:
$$\exists \beta > \gamma_n \; (\beta \in A_n).$$
This statement holds in  $\mathcal{H}_\theta$, so by elementarity, it also holds in $M$. Thus, there exists $\beta_n \in M$ with $\beta_n > \gamma_n$ and $\beta_n \in A_n$. This completes the  construction.

\p We have produced a sequence $\left<\beta_n \: : \: n< \omega\right>$ with $\beta_n \in M$ for all $n$. But clearly  $\sup_n \beta_n = \sup_n \gamma_n = \delta \in S$, so Player II wins this game, proving that the strategy was not winning for Player I.
\end{proof}

\begin{Def} \label{short} A \emph{short $\kappa$-Laver tree} is a tree $L \subseteq \kappa^{<\omega}$ (i.e., height $\omega$), such that for all  $ \sigma \in L$ extending $\stem(L)$ we have $ |\Succ_L(\sigma)| = \kappa$. \end{Def}

\begin{Cor} \label{ccor} Let $S \subseteq \kappa$ be a stationary subset of $\Cof_\omega(\kappa)$. For every short $\kappa$-Laver tree $L$ there exists a branch $\eta \in [L]_\omega$ such that $\sup_n \eta(n) \in S$. 
\end{Cor}

\begin{proof} The short $\kappa$-Laver tree $L$ induces a strategy $\sigma_L$ for Player I in the supremum game:  $$\sigma_L(A_0, \beta_0, \dots, A_n, \beta_n) := \Succ_L(\stem(L) \cc \left<\beta_0, \dots, \beta_n\right>).$$ Whenever $\left<A_0, \beta_0, A_1, \beta_1, \dots \right>$ is a run of the game according to $\sigma_L$, $\stem(L) \cc \left<\beta_0, \beta_1, \dots \right>$ is an element of $[L]_\omega$. 

\p By Lemma \ref{gamelemma}, there exists a run  of the game in which Player I follows   $\sigma_L$ but Player II wins. This yields a branch $\eta \in [L]_\omega$ such that $\sup_n \eta(n) \in S$. \end{proof}

With this, we immediately obtain our main result.

\begin{Thm}[Main Theorem 1] \label{Mainthm} Let $\IP$ be any forcing whose conditions are $\kappa$-Laver trees $($i.e., $\IP \subseteq \IL_\kappa)$  and which is closed under the following condition: if $T \in \IP$ and $\sigma \in T$, then $T \thru \sigma \in  \IP$. Then $\IP$ adds a Cohen $\kappa$-real. \end{Thm}

\begin{proof}

We will use the following notation: if $T \in \klk$ is a tree and $\sigma \in T$, then $T \till^\omega \sigma := \{\tau \in \kappa^{{<}\omega}  :  \sigma \cc \tau \in T\}.$ Note that if $T$ is a $\kappa$-Laver tree, then for every $\sigma \in T$ extending $\stem(T)$, $T \till^\omega \sigma$ is a short $\kappa$-Laver tree (with empty stem).

\p
Let $S_0 \cup S_1$ be a stationary/co-stationary partition of  $\Cof_\omega(\kappa)$ and consider the mapping
$ \varphi: \kk \; \rightarrow \; \dk$ defined by $$\varphi(x)(\alpha) = 1 \;\; :\Leftrightarrow \; \; \sup\{x(\omega {\cdot} \alpha + n) \: : \: n<\omega\} \in S_1.$$ In other words, partition $x$ into $\kappa$-many blocks of length $\omega$, and map each piece to $0$ or $1$ depending on whether its supremum lies in $S_0$ or $S_1$. We claim that if $\xg$ is $\IP$-generic then $\varphi(\xg)$ is $\kappa$-Cohen-generic.

\p We use $\tilde{\varphi}: \klk \to \dlk$ to denote the approximations of $\varphi$ (defined as above). Let $T \in \IP$ be given and let $D$ be open dense in $\kappa$-Cohen forcing. Let $\sigma := \stem(T)$, w.l.o.g. $\len(\sigma)$ is a limit ordinal. Let $t \in D$ extend $\tilde{\varphi}(\sigma)$. Suppose  $\tilde{\varphi}(\sigma) \cc \left<0\right> \subseteq t$. By Corollary \ref{ccor} there  is $\eta \in [T \till^\omega  \sigma]_\omega$ such that $\sup_n \eta(n) \in S_0$. If, instead, we have $\tilde{\varphi}( \sigma) \cc \left<1\right> \subseteq t$, we can apply Corollary \ref{ccor} and find a branch $\mu \in [T \till^\omega \sigma]_\omega$ such that $\sup_n \mu(n) \in S_1$. Note that, since $T$ is limit-closed,   $\sigma \cc \eta$ resp. $\sigma \cc \mu$ are elements of $T$. Now proceed analogously until reaching $\tau$, such that $\tilde{\varphi}(\tau) = t$. By assumption $T \thru \tau \in \IP$, and now clearly $T \thru \tau \Vdash \tau \subseteq \xxg$ and therefore $T \thru \tau \Vdash t \subseteq \varphi(\xxg)$. Thus  $\varphi(\xg)$ is  a Cohen $\kappa$-real.\end{proof}

Another way of looking at the above proof is as follows: the sets $\{\eta \in \kappa^\omega : \sup_n \eta(n) \in S_0\}$ and $\{\eta \in \kappa^\omega : \sup_n \eta(n) \in S_1\}$ form Bernstein sets with respect to short $\kappa$-Laver trees in $\kappa^{<\omega}$. Note that due to cardinality reasons, we cannot use standard diagonalization arguments to produce such sets.

\bigskip If we additionally assume $\klk = \kappa$, we can obtain an even stronger theorem. 

\begin{Def} \label{pseudo} A tree $T \subseteq \klk$ is called a \emph{pseudo-$\kappa$-Laver tree} if it is limit-closed and has the following property: every $ \sigma \in T$ has an extension $\tau \in T$ such that $ T \till^\omega \tau$ is a short $\kappa$-Laver tree. We use $\PIL_\kappa$ to denote the partial order of pseudo-$\kappa$-Laver trees ordered by inclusion. \end{Def}


\begin{Thm}[Main Theorem 2] \label{Mainthm2} Assume $\klk=\kappa$. Let $\IP$ be any forcing whose conditions are pseudo-$\kappa$-Laver trees $($i.e., $\IP \subseteq \PIL_\kappa)$  and which is closed under the following condition:  if $T \in \IP$ and $\sigma \in T$, then $T \thru \sigma \in  
\IP$. Then $\IP$ adds a Cohen $\kappa$-real. \end{Thm}

\begin{proof} The method is similar, except that now  we let $\{S_t \: : \: t \in \klk\}$ be a  partition of $\Cof_\omega(\kappa)$ into $\kappa$-many disjoint stationary sets, which we index   by $\klk$. This is possible due to the assumption $\klk = \kappa$. Define the mapping $\pi: \kk \to \dk$ by $\pi(x) := t_0 \cc t_1 \cc t_2 \cc \dots$,
where for all $\alpha<\kappa$, $t_\alpha$ is such that $\sup\{	x(\omega \cdot \alpha + n) \: : \: n<\omega\} \in S_{t_\alpha}$. We also use $\tilde{\pi}$ to denote the same operation but from $\klk$ to $\dlk$.


\p Let $\xg$ be the $\IP$-generic $\kappa$-real;  we show that  $\pi(\xg)$ is $\kappa$-Cohen. Let $D$ be open dense in $\kappa$-Cohen forcing, and let $T \in \IP$. Find $\sigma \in T$ such that $T \till^\omega \sigma$ is a short $\kappa$-Laver tree. Let $t \in D$ be such that $\tilde{\pi}(\sigma) \subseteq t$. Let $u$ be such that $\tilde{\pi}(\sigma) \cc u = t$. By Corollary \ref{ccor} there is $\eta \in [T \till^\omega \sigma]_\omega$ such that $\sup_n \eta(n) \in S_u$. It follows that $\tilde{\pi}(\sigma \cc \eta) = \tilde{\pi}(\sigma) \cc u =  t$. Therefore $T \thru (\sigma \cc \eta) \Vdash t \subseteq \pi(\xxg)$. \end{proof}

\bigskip
\section{The generalized Laver dichotomy} \label{goldstern}

The supremum game and the arguments from Theorem \ref{Mainthm} naturally lead us to consider a question in generalized descriptive set theory (this connection is explained in Remark \ref{relevance}). 

We need  the  following strengthening of the concept of a \emph{dominating real}, which has  been studied in the classical context in \cite{GoldsternGame, LabedzkiRepicky, RepickyDeco, FilterLaver}.

\begin{Def}  For $f: \klk \to \kappa$ and $x \in \kk$, we say that \emph{$x$ strongly dominates $f$} if $\exists \alpha_0 \; \forall \alpha > \alpha_0 \;(x(\alpha) \geq f(x \till \alpha))$. If $M$ is a model of set theory with the same $\klk$, then $x$ is called \emph{strongly dominating over $M$} if for all $f: \klk \to \kappa$ with $f \in M$, $x$ strongly dominates $f$.
\end{Def}

Clearly, if $x$ is strongly dominating, then it is  also dominating. The converse is false in general, e.g., let $d$ be dominating over $M$ and let $x$ be defined by $x(\alpha) := d(\alpha)$ for odd $\alpha$ and $x(\alpha) := d(\alpha+1)$ for even and limit $\alpha$. Then $x$ is dominating but not strongly dominating. However, the following is true:

\begin{Lem} \label{strongly} Assume  $\klk=\kappa$. Let $M$ be a model of set theory such that $\klk \cap M  = \klk$. Then, if there is a dominating $\kappa$-real over $M$ there  is also a strongly dominating $\kappa$-real over $M$.  \end{Lem}

\begin{proof} Let $d$ be the dominating $\kappa$-real, and fix a bijection between $\klk$ and $\kappa$ in $M$. We can define a new dominating $\kappa$-real $d^*:\klk \to \kappa$, i.e., such that for  every $f: \klk \to \kappa$ in $M$, $f(\sigma) \leq d^*(\sigma)$ holds for all but ${<}\kappa$-many $\sigma \in \klk$. Now define inductively $$e(\alpha) := d^*(e \till \alpha).$$ 

\p Then $e$ is strongly dominating. \end{proof}



\begin{Def} A collection $X \subseteq \kk$ is  a \emph{strongly dominating family} if for every $f:\klk \to \kappa$ there exists $x \in X$ which strongly dominates $f$. $\mathcal{D}_\kappa$ denotes the ideal of all $X \subseteq \kk$ which are \emph{not} strongly dominating families.\end{Def}

For $\kappa = \omega$, the ideal $\mathcal{D}_\omega = \mathcal{D}$ is the well-known \emph{non-strongly-dominating ideal}, introduced in   \cite{GoldsternGame} and independently in \cite{Za04}, and studied among others in \cite{RepickyDeco}. The main interest in it stems from a perfect-set-like dichotomy theorem for Laver trees.

\begin{Thm}[Goldstern et al.\  \cite{GoldsternGame}]  If $T \subseteq \wlw$ is a Laver tree then $[T] \notin \D$. Every  analytic set $A \subseteq \ww$ is either in $\D$ or contains $[T]$ for some Laver tree $T$. In particular, there is a dense embedding from the order of  Laver trees into the  algebra of Borel subsets of $\ww$ modulo $\D$. \end{Thm}

Dichotomies such as this one are common in classical descriptive set theory, the most notable example being the perfect set property and the closely related $K_\sigma$-dichomoty (\cite{KechrisDichotomy}), all of which are false for arbitrary sets of reals but true for analytic sets. Interest in generalizing such dichotomies to the $\kk$-context was recently spurred by a result of  Schlicht \cite{SchlichtPSP} showing that the generalized perfect set property for generalized projective sets is consistent, and  L\"ucke-Motto Ros-Schlicht \cite{LuckeHurewicz}  showing that the   generalized Hurewicz dichotomy   for generalized projective sets is consistent. Thus, it might initially seem surprising that the generalized Laver dichotomy fails for closed sets, provably in ZFC.

\begin{Thm} \label{mainprop} There is a closed subset of $\kk$ which is neither in $\D_\kappa$ nor contains the branches of a generalized Laver tree. \end{Thm}

\begin{proof} Let  $\varphi$ be as in the proof of Theorem \ref{Mainthm}. Let $z$ be the constant $0$ function (or any other fixed element of $ \dk$). We show that $C := \varphi^{-1}[\{z\}]$ is a counterexample to the dichotomy. Given any $T \in \IL_\kappa$, we can easily find $x \in [T]$ such that $\varphi(x) \neq z$, therefore $[T] \not\subseteq C$. We claim that $C$ is strongly dominating. Let $f: \klk \to \kappa$ be given. Let
$$T_f := \{\sigma \in \klk \: : \: \forall \beta < \len(\sigma) \; (\sigma(\beta) \geq f(\sigma \till \beta))\}.$$
Clearly $T_f$ is a generalized Laver tree and  $\stem(T_f) = \varnothing$. As in the proof of Theorem \ref{Mainthm},  we can find  $x \in [T_f]$ such that $\varphi(x) = z$. But then $x$ strongly dominates $f$ and $x \in C$, completing the argument. 
\end{proof}

\begin{Remark} \label{relevance} The relevance of this lemma  is that it explains why Theorem \ref{Mainthm} does not  (as one might initially assume) yield a ZFC-proof of $\bb_\kappa \leq \cov(\M_\kappa)$. Indeed, it is not hard to verify that $\cov(\D_\kappa) = \bb_\kappa$ and that   if $X \in 
\M_\kappa$  then  $\varphi^{-1}[X]$ does not contain a $\kappa$-Laver tree. Thus, if the dichotomy would hold for generalized Borel (or just $F_\sigma$) sets then  one could have concluded  $\bb_\kappa = \cov(\D_\kappa) \leq \cov(\M_\kappa)$.
\end{Remark}

One could wonder whether there is \emph{any} dichotomy for the ideal $\D_\kappa$, i.e., whether there is any collection $\IP$ of limit-closed trees, such that for every $T \in \IP$, $[T] \notin \D_\kappa$, and every analytic (or at least closed) set not in $\D_\kappa$ contains $[T]$ for some $T \in \IP$. In fact, this is not the case either.


\begin{Lem}  \label{pseudogame} Let $T \subseteq \klk$ be a tree such that $[T]$ is strongly dominating. Then there exists $s \in T$ such that $T \till^\omega s$ contains a short $\kappa$-Laver tree. \end{Lem}

\begin{proof} We use a slightly modified version of the game  from    \cite{GoldsternGame}. Given $A \subseteq \kappa^\omega$ let $G^{\star}(A)$ be the game defined by:

$$\begin{array}{r||cccccc}
{\rm I} & \alpha_0 & & \alpha_1 & & \dots \\ \hline
{\rm II} & & \beta_0 & & \beta_1 & & \dots \end{array}$$
where  $\alpha_n, \beta_n < \kappa$, $\alpha_n \leq \beta_n$ for all $n$, and  Player II wins   iff $\left<\beta_n \: : \: n<\omega\right> \in A$. 

\p
It is easy to see that if Player II has a winning strategy in $G^{\star}(A)$ then there exists a short $\kappa$-Laver tree $L$ (with empty stem) such that $[L]_\omega \subseteq A$. Also it is well-known and easy to see that if $A$ is closed (in the topology on $\kappa^\omega$) then  $G^{\star}(A)$ is determined.

\p Suppose, towards contradiction, that there is no $s \in T$ such that $T \till^\omega s$ contains a short $\kappa$-Laver tree. Then Player II does not have a winning strategy in $G^{\star}([T \till^\omega s]_\omega)$ for any $s \in T$, and therefore Player I has a winning strategy, call it $\sigma_s$. Define $f: \klk \to \kappa$ as follows: for every $t \in T$, let $s \subseteq t$ be the maximal node of limit length, let $u$ be such that $t = s \cc u$,  and define $f(t) := \sigma_{s}(u)$. Since $[T]$ is strongly dominating there is $x \in [T]$ and $\alpha$ such that $x(\beta) \geq f(x \till \beta)$ for all $\beta>\alpha$. In particular, there is $s \subseteq x$, of limit length, such that $x(|s| + n) \geq f(x \till (|s| + n))$ for all $n<\omega$. Letting $z \in \kappa^\omega$ be such that $s \cc z  = x \till (|s|+\omega)$, we see that $z(n) \geq f(s \cc z \till n) = \sigma_s(z \till n)$, for every $n$. Also $z \in [T \till^\omega s]_\omega$, therefore $z$ satisfies the winning conditions for Player II in the game $G^{\star}([T \till^\omega s]_\omega)$, contradicting the assumption that $\sigma_s$ was a winning strategy for Player I. \end{proof}



\begin{Cor} \label{corzonneveld} There exists a closed strongly dominating set without a super-closed strongly dominating subset. \end{Cor}

\begin{proof} Consider again the closed set $C := \varphi^{-1}[\{z\}]$ from the proof of  Theorem \ref{mainprop}. Towards contradiction suppose there is a limit-closed tree $T$ such that $[T] \subseteq C$ and $[T]$ is strongly dominating. Without loss of generality, we may assume that $T$ is \emph{pruned}, in the sense that for every $s \in T$ there is a proper extension $t \in T$.

\p By Lemma \ref{pseudogame} there is $s \in T$ such that  $T \till^\omega s$ contains a short $\kappa$-Laver tree $L$. By Corollary \ref{ccor} there is $\eta \in [L]_\omega$ such that $\sup_n \eta(n) \in S_1$, and by limit-closure, $s \cc \eta \in T$. Moreover, since $T$ is limit-closed and pruned, there is $x \in [T]$ such that $s \cc \eta \subseteq x$. But then $\varphi(x)$ contains a ``$1$'' and thus is not equal to $z$, the constant $0$-function, contradicting the assumption. \end{proof}

Lemma \ref{pseudogame}, whose proof is based on the game argument from   \cite{GoldsternGame}, will be an important ingredient in the following section.

\bigskip
\section{${<}\kappa$-distributive tree forcings} \label{sectionbetter}

We would like to generalize the results from  Section \ref{game}  about Laver trees to a wider class of forcing notions. Recall that a forcing $\IP$ is \emph{${<}\kappa$-closed} if for every decreasing sequence of conditions of length ${<}\kappa$, there is a condition below all of them.  A forcing $\IP$ is \emph{${<}\kappa$-distributive} if the intersection of ${<}\kappa$-many open dense sets is open dense. Since  ${<}\kappa$-distributive forcings do not add new elements of $\klk$, it is a  natural class to consider  in the context of generalized Baire spaces (after all, forcing in the ordinary Baire space does not add new finite sequences). If a forcing is ${<}\kappa$-closed, then it is ${<}\kappa$-distributive, although the converse does not hold. One interesting difference between the two, in the context of generalized descriptive set theory, is that generalized-$\boldsymbol{\PI}^1_1$-absoluteness holds between ${<}\kappa$-closed forcing extensions (see \cite[Lemma 2.7]{GeneralizedVadim}), while it may fail for ${<}\kappa$-distributive forcing extensions. In this sense, the most natural question is the following:

\begin{Question} \label{qqq} Is it true that every ${<}\kappa$-distributive forcing adding a dominating $\kappa$-real adds a Cohen $\kappa$-real? Is it at least true for every ${<}\kappa$-closed forcing? \end{Question}

Although we cannot  answer this question in generality, we can answer the question for ${<}\kappa$-distributive forcings whose conditions are limit-closed trees, and such that a dominating $\kappa$-real can be defined from the generic by a ground-model continuous function. More generally, this holds whenever the \emph{interpretation tree} of the dominating $\kappa$-real is limit-closed.

In this section, we will always assume that $$\klk = \kappa.$$



\begin{Def}  Let $\IP$ be any forcing notion, let $\dot{x}$ be a name, and let $p \in \IP$ be such that $p \Vdash \dot{x} \in \kk$. Then the \emph{interpretation tree} of $\dot{x}$ below $p$ is defined by:
$$\T_{\dot{x},p} = \{\sigma \in \klk \: : \: \exists q \leq p \; (q \Vdash \sigma \subseteq \dot{x})\}.$$
\end{Def}

It is clear that $\Txp$ is always a tree in the ground model, but in general it need not be a limit-closed tree.  

\begin{Lem} \label{lemma1} Suppose $\IP$ is a ${<}\kappa$-distributive forcing, and suppose $p \Vdash $ ``$\dot{d}$ is a strongly dominating $\kappa$-real''. Additionally, assume that for every $q \leq p$, the interpretation tree $\Tdq$ is limit-closed. Then  $p \Vdash $ ``there is a Cohen $\kappa$-real''. \end{Lem}

\begin{proof} Let $\pi$ be the function defined in Theorem \ref{Mainthm2}. We will show that $p \Vdash $``$\pi(\dot{d})$ is $\kappa$-Cohen''.  Let $D$ be $\kappa$-Cohen dense and $q \leq p$ arbitrary.  

\p \textbf{Claim:} $[\Tdq]$ is a strongly dominating set.

\begin{proof} Let $f: \klk \to \kappa$. Since $q$ forces that $\dot{d}$ is strongly dominating, in particular $q \Vdash \exists \beta \: \forall \alpha > \beta \: ( \dot{d}(\alpha) \geq \check{f}(\dot{d}\till \alpha))$. By ${<}\kappa$-distributivity, there is a $\beta_0$ and $q_0 \leq q$ which decides $\dot{d} \till \beta_0 =: \sigma_0$ and forces the following:


$$\forall \alpha > \beta_0 \: (\dot{d}(\alpha) \geq \check{f}(\dot{d}\till \alpha)). \;\;\;\;\;\;\;\;\;\;\;\; (*)$$
Consider the interpretation tree $\T_{\d, q_0}$. Let $x$ be any branch in $[\T_{\d, q_0}] \subseteq [\Tdq]$. To see that such a branch exists, notice that for any $\sigma \in \T_{\d, q_0}$ there is a condition $q'$ deciding $\sigma \subseteq \dot{d}$, and by ${<}\kappa$-distributivity, we can find a stronger condition $q'' \leq q'$ deciding $\tau \subseteq \dot{d}$ for a proper extension $\tau$ of $\sigma$. Moreover, at limit nodes we can continue  since $ \T_{\d, q_0}$ is limit-closed by assumption. 

Now we see that  for any initial segment $\sigma \subseteq x$ which is longer than $\sigma_0$, we know that some $q' \leq q_0$ forces $\sigma \subseteq \d$. Since $q'$ also forces $(*)$, we must have $\sigma(\alpha) \geq f(\sigma \till \alpha)$ for all $\alpha$ in the domain of $\sigma$ with $\alpha > \beta_0$. Thus we conclude that   $x(\alpha) \geq f(x\till \alpha)$ holds for every $\alpha > \beta_0$. \qedhere (Claim)

\end{proof}

\p From the Claim and  Lemma \ref{pseudogame}, it follows that there is $\sigma \in \Tdq$ such that $\Tdq \till^\omega \sigma$ contains a short $\kappa$-Laver tree. Just as in the proof of Theorem \ref{Mainthm2},  let $t \in D$ be such that $\tilde{\pi}(\sigma) \subseteq t$,   $u$   such that $\tilde{\pi}(\sigma) \cc u = t$, and find  $\eta \in [\Tdq \till^\omega \sigma]_\omega$ such that $\sup_n \eta(n) \in S_u$. Now, notice that by the assumption that $\Tdq$ is limit-closed, $\sigma \cc \eta \in \Tdq$, hence there is $r \leq q$ forcing $\sigma \cc \eta \subseteq \dot{d}$. But then  $$r \:  \Vdash  \: t= \tilde{\pi}(\sigma) \cc u   = \tilde{\pi}(\sigma \cc \eta) \subseteq \pi(\d),$$ and so $r \: \Vdash \: \pi(\d) \in [t]$. 
 \end{proof}



\bigskip Next we look at  forcings $\IP$ whose conditions are limit-closed trees on $\klk$. 

\begin{Def} A forcing partial order $\IP$ is called a \emph{tree forcing} if its conditions are limit-closed trees $T \subseteq \klk$, and for every $T \in \IP$ and $\sigma \in T$, the restriction $T \thru \sigma \in \IP$.  \end{Def}

We  need to review continuous functions on $\kk$.  Let us call a function $h: \klk \to \klk$ \emph{pre-continuous} if: \begin{enumerate}
\item $\sigma \subseteq \tau \; \Rightarrow \; h(\sigma) \subseteq h(\tau)$, and  
\item $\forall x \in \kk, \;   \{ \len(h(\sigma)) \: : \: \sigma \subseteq x \}$ is cofinal in $\kappa$. \end{enumerate}

\p If $h$ is pre-continuous, let $f = \lim(h)$ be the function defined as  $f(x) := \bigcup \{h(\sigma) : \sigma \subseteq x\}$. Just as in the classical situation, it is easy to check that if   $h$ is pre-continuous, then $\lim(h)$  is continuous, and for every continuous $f$   there exists a pre-continuous $h$ such that $f = \lim(h)$. 

Unlike the classical situation, ``being pre-continuous'' is not necessarily an absolute notion. The statement (2) above is a generalized-$\PI^1_1$-statement, so it will be absolute between ${<}\kappa$-closed forcing extensions, but not necessarily between arbitrary ${<}\kappa$-distributive forcing extensions. However  in our case, this will not present a problem. We will always talk about pre-continuous functions in the ground model, and implicitly assume that the continuous function in the extension is well-defined at least on the generic $\kappa$-real. 

\newcommand{\tr}{{\rm tr}}
\newcommand{\pf}{''}

\bigskip The main point is that for tree forcings, the interpretation trees are directly related to the forcing conditions.  For a tree $T$ and a pre-continuous function $h$, we will  consider the tree generated by the image of $T$ under $h$:
$$\tr(h\pf T) := \{\tau \: : \: \exists \sigma \in T \: ( \tau \subseteq h(\sigma))\}.$$

\begin{Lem} \label{lemma3} Let $\IP$ be a ${<}\kappa$-distributive tree forcing, $\dot{x}$ a name for a $\kappa$-real,  $h$ a pre-continuous function in the ground model with $f = \lim(h)$,  and suppose that $T \in \IP$ is such that $T \Vdash \dot{x} = f(\xxg)$.\footnote{In particular, part of this assumption is that $T$ forces that $\{ \len(h(\sigma)) :  \sigma \subseteq \xxg \}$ is cofinal in $\kappa$. Recall that even if $h$ is pre-continuous in the ground model, the second condition may fail to be absolute. If $\IP$ is ${<}\kappa$-closed, then the condition is preserved by  $\PI^1_1$-absoluteness.}   Then $\T_{\dot{x},T} = \tr(h\pf T)$. \end{Lem}

\begin{proof} First suppose $\sigma \in T$. Then $T \thru \sigma \Vdash \sigma \subseteq \xxg$, therefore $T \thru \sigma \Vdash h(\sigma) \subseteq f(\xxg) = \dot{x}$. Therefore $h(\sigma) \in \T_{\dot{x}, T}$.

\p Conversely, let $\tau \in \T_{\dot{x}, T}$ be given. We want to find $\sigma \in T$ such that $\tau \subseteq h(\sigma)$. By definition there is $S \leq T$ such that $S \Vdash \tau \subseteq \dot{x}$. But since $S \Vdash \dot{x} = f(\xxg)$, we also have $$S \Vdash \exists \sigma \subseteq \xxg \: (\tau \subseteq h(\sigma)).$$

\p By ${<}\kappa$-distributivity, there exists $S' \leq S$ which decides $\sigma$, i.e., we may assume that $\sigma$ is in the ground model, $\tau \subseteq h(\sigma)$ holds, and $S' \Vdash \sigma \subseteq \xxg$. Moreover, $\sigma  \subseteq \stem(S')$, because otherwise there would be some incompatible $\sigma' \in S'$, and we would have $S' \thru \sigma' \Vdash \sigma' \subseteq \xxg$, contradicting $S' \Vdash \sigma \subseteq \xxg$. We conclude that $\sigma  \in S' \subseteq S \subseteq T$ and $\tau \subseteq h(\sigma)$ as desired.  \end{proof}

Taking $h$ to be the identity, an immediate corollary is that if $\IP$ is a ${<}\kappa$-distributive tree forcing, then the interpretation trees for the generic $\xxg$ are limit-closed. If, in addition,  the generic is strongly dominating, then by Lemma \ref{lemma1} we immediately know that $\IP$ adds Cohen $\kappa$-reals. 

\bigskip For our stronger result, we want to consider  pre-continuous functions $h$ other than the identity. In those cases, it is not guaranteed that $\tr(h \pf T)$ is limit-closed, even if $T$ was. To avoid this problem we prove two technical lemmas. The main idea is that, even if the original continuous function does not preserve limit-closure, we may change it to another one which does. 

\begin{Def} A  pre-continuous  function $h$ is called  \emph{limit-closure-preserving} if for every limit-closed tree $T$, the tree $\tr(h \pf T)$ is also limit-closed.  
\end{Def}

\begin{Lem} \label{higher} For every pre-continuous function $h$, there exists a pre-continuous and limit-closure-preserving function  $j$, such that for all $\sigma$ and all $\alpha$ (in the respective domains), we have:
$$h(\sigma)(\alpha) \leq j(\sigma)(\alpha).$$ \end{Lem}

\smallskip

\begin{proof} Fix a function $R: \klk \times \klk \to \klk$ such that: \begin{enumerate}
\item $R(\rho,\varnothing) = \varnothing $ for all $\rho$.
\item If $\sigma \neq \varnothing$, then \begin{itemize}
\item $\len(R(\rho,\sigma)) = \len(\sigma)$ for all $\rho$,  

\item $\sigma(\alpha) \leq R(\rho,\sigma)(\alpha)$ for all $\rho$ and all $\alpha < \len(\sigma)$. \end{itemize}

\item If $\rho \neq \rho'$, then for any $\sigma, \sigma' \neq \varnothing $, we have $R(\rho,\sigma)(0) \neq R(\rho',\sigma')(0) $.  \end{enumerate}

\p In words: $R$ takes every non-empty sequence $\sigma$ and shifts it coordinate-wise  to a higher sequence of the same length depending on $\rho$; this happens in such a way that for different $\rho \neq \rho'$, the first coordinates of $R(\rho, \dots)$ and $R(\rho', \dots)$ are never the same.   It is easy to see that such a function exists since $\klk = \kappa$.

\p Let $h$ be a pre-continuous function. Define $j$ inductively: \begin{itemize}

\item If $j(\sigma)$ is defined, then for  every $\beta$ define $j(\sigma \cc \left<\beta\right>)$ as follows: let $w$ be such that $h(\sigma) \cc w = h(\sigma \cc \left< \beta \right> )$ ($w = \varnothing$ is also allowed). Then let $$j(\sigma \cc \left<\beta\right>) := j(\sigma) \cc R(\sigma \cc \left<\beta\right>, w).$$

\item For $\sigma$ of limit length (including $\sigma=\varnothing)$, let $w$ be such that $h(\sigma) = \bigcup_{\sigma' \subset \sigma} h(\sigma') \cc w$. Note that this is always possible because $h(\sigma') \subseteq h(\sigma)$ for all $\sigma'\subset \sigma$ ($w=\varnothing$ is allowed). Then let
$$ j(\sigma) := \left( \: \bigcup_{\sigma' \subset \sigma} j(\sigma') \right) \cc R(\sigma, w).$$

\end{itemize}


\p We claim that $j$ is as required.

\p Notice that, inductively,  $\len(j(\sigma)) = \len(h(\sigma))$ for every $\sigma$.  It is also clear, by construction, that $\sigma \subseteq \sigma'$ implies $j(\sigma) \subseteq j(\sigma')$. Therefore $j$ is pre-continuous.  Moreover, by construction we immediately see that  $h(\sigma)(\alpha) \leq j(\sigma)(\alpha)$ holds for every $\sigma$ and $\alpha  < \len(\sigma)$. It remains to prove that $j$ is limit-closure-preserving.

\p Let $T$ be an arbitrary limit-closed tree, and let $U := \tr(j \pf T)$. Let $\{u_i : i<\lambda\}$ be an increasing sequence in $U$ of length $\lambda<\kappa$. We need to show that this sequence has an extension in $U$.  For each $i$, let $s_i \in T$ be \emph{$\subseteq$-minimal} such that $u_i \subseteq j(s_i)$.\footnote{The $s_i$'s do not need to be distinct; e.g., they could be all equal to a unique $s$, or there could be $\cf(\lambda)$-many distinct $s_i$'s, etc.}

\p \textbf{Claim.} $s_{i} \subseteq s_{i'}$ for all $i<i'<\lambda$.


\begin{proof}  Suppose, towards contradiction, that $s_i \not\subseteq s_{i'}$.  First, $s_{i'} \subset s_i$ (proper extension) is clearly not possible, since this would imply $u_i \subseteq u_{i'} \subseteq j(s_{i'}) \subseteq j(s_i)$, and thus we would have picked $s_{i'}$ instead of $s_i$. Therefore, $s_i$ and $s_{i'}$ are incompatible. Let $r$ be maximal such that $r \subseteq s_i$ and $r \subseteq s_{i'}$.

\p Next, notice that $j(r) \subset u_i$: otherwise, we would have $u_i \subseteq j(r)$, so we would have  picked $r$ instead of $s_i$.

\p So we also know that $j(r) \subset j(s_i)$ and $j(r) \subset j(s_{i'})$. Let $r_0$ be minimal such that 
$$r \subseteq r_0   \subseteq s_i \;\;\;\; \text{and} \;\;\;\; j(r)  \subset j(r_0)  $$
and let $r_1$ be minimal such that
$$r \subseteq r_1   \subseteq s_{i'} \;\;\;\; \text{and} \;\;\;\; j(r) \subset j(r_1).  $$ Note that both $r_0$ and $r_1$ are proper extensions of $r$, see Figure \ref{diagram}. First we consider $r_0$: there are two cases. \begin{itemize}
\item Suppose $r_0$ is of successor length. Then there is $r_{00}$ such that $r_0 = r_{00} \cc \left<\beta\right>$ and $j(r) = j(r_{00})$. Also, (since $j(\sigma)$ and $h(\sigma)$ always have the same length),  there exists   $w\neq \varnothing$  such that $h(r_{00} \cc \left<\beta\right>) = h(r_{00}) \cc w$. Then by definition we have: 
$$j(r_0) \;\;  =\;\;  j(r_{00}) \cc R(r_0, w)  \; \; =\;\;   j(r) \cc R(r_0, w). $$

\item Now suppose $r_0$ is of limit length. Then $j(r) = j(r')$ for all $r'$ with $r \subseteq r' \subset r_0$, but $h(r_0) \supset \bigcup_{r' \subset r_0} h(r')$. So (again because $j(\sigma)$ and $h(\sigma)$ have the same length) there exists $w \neq \varnothing$ such that $h(r_0) =  \bigcup_{r' \subset r_0} h(r') \cc w$. By definition, we have

$$j(r_0) \;\;  =\;\;   \left( \bigcup_{r' \subset r_0} j(r')  \right) \cc  R(r_0, w)  \; \; =\;\;   j(r) \cc R(r_0, w). $$
\end{itemize}

\p Thus, in both cases we have $j(r_0)   =    j(r) \cc R(r_0, w)$ for some non-empty $w$. 

\p By exactly the same argument but looking at $r_1$, we  see that $j(r_1) = j(r) \cc R(r_1, v)$ for some non-empty $v$.

\p But $r_0 \neq r_1$,  so by   condition 3 of the definition of $R$, the first coordinates of $R(r_0, w)$ and of $R(r_1, v)$ are not the same. However, we also know  $ j(r) \cc R(r_0, w) \subseteq j(s_i)$ while $  j(r) \cc R(r_1, v) \subseteq j(s_{i'})$. Together with the fact that   $j(r) \subset u_i \subseteq j(s_i)$ and $j(r) \subset u_i \subseteq u_{i'} \subseteq j(s_{i'})$, this gives us the desired contradiction (see Figure \ref{diagram}). We conclude that the only option is $s_i \subseteq s_{i'}$. \qedhere (Claim)

\begin{figure}[h]
\centering
\includegraphics[width=13cm]{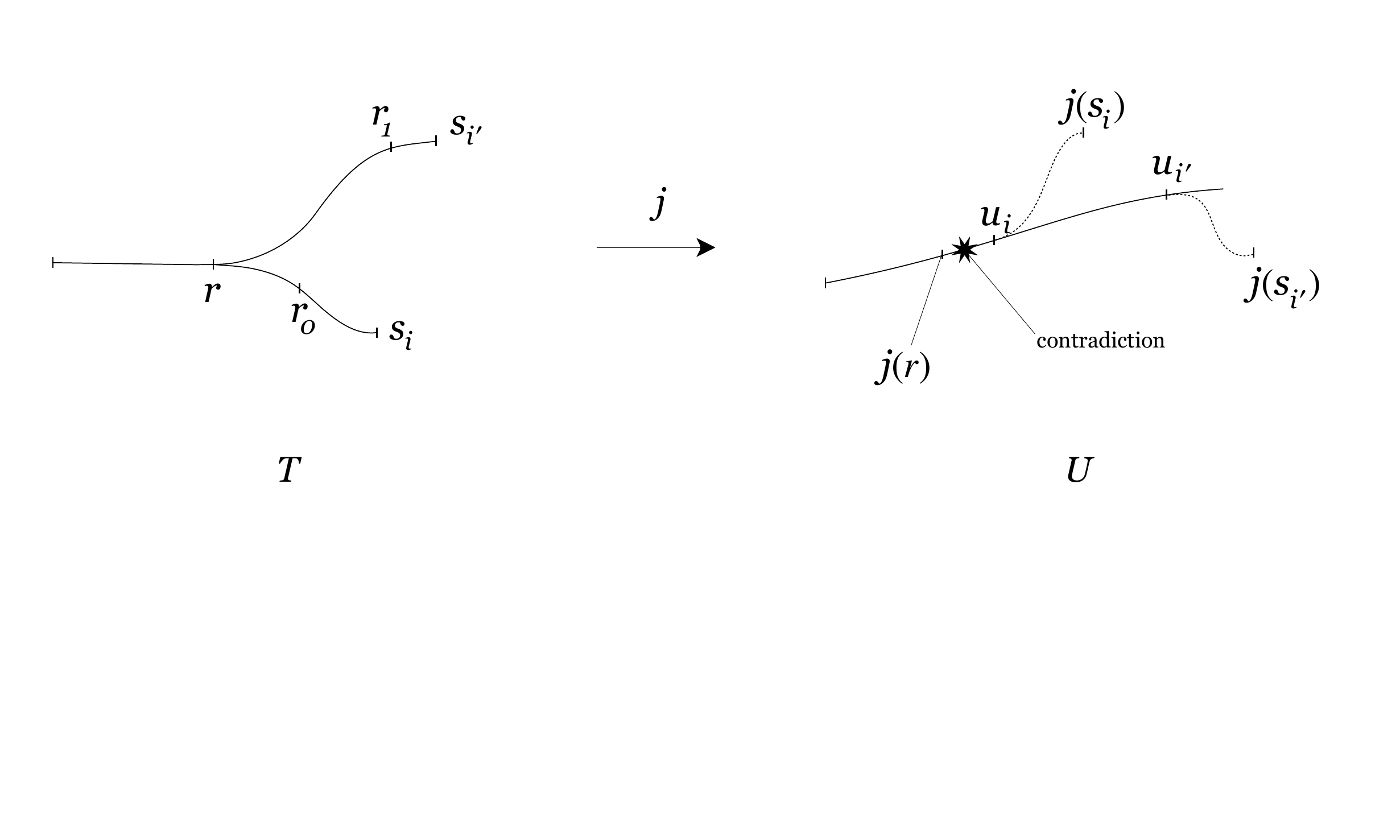}
 \vskip-3cm
\caption{Contradiction assuming  $s_i \;\bot\; s_{i'}$}
\label{diagram}
\end{figure}
\end{proof}

\p So we have an increasing sequence $\{s_i \: : \: i<\lambda\}$ in $T$, and since $T$ is limit-closed, there is $s_\lambda \in T$ with $s_i \subseteq s_\lambda$ for all $i$. Then $u_i \subseteq j(s_i) \subseteq j(s_\lambda)$ holds for all $i$. This completes the proof that $U$ is limit-closed. \end{proof}

The point of this lemma is that if $h$ is pre-continuous  in the ground model with $f = \lim(h)$ and  $T$ forces that $f(\xxg)$ is a dominating $\kappa$-real, then letting $j$ be as in the lemma with $g=\lim(j)$, we know that $T$ also forces that  $g(\xxg)$ is a dominating $\kappa$-real. 

The next step is to convert the dominating into a strongly dominating real.  In Lemma \ref{strongly} we mentioned how to convert a dominating to a strongly dominating real, and it is easy to see that this conversion can be coded by a continuous function in the ground model. The problem is, this function may again fail to be limit-closure-preserving, so we need to use a similar method as above to construct such a conversion function which is, in addition, limit-preserving.

Let us fix an enumeration $\{\sigma_i : i < \kappa\}$ of $\klk$ such that $\sigma_i \subseteq \sigma_j \: \Rightarrow i < j$, using the  notation $\q{\sigma} = i$ iff $\sigma = \sigma_i$. Recall that in Lemma \ref{strongly}, the conversion was given by $e(\alpha) = d^*(e \till \alpha) = d(\q{e \till \alpha})$. However, we may relax the condition to $e(\alpha) \geq d(\q{e \till \alpha})$, and the conversion would still work. 

\begin{Def} A function $\gamma: \kk \to \kk$ is called \emph{strongly-converting}, if for all $x$  and all $\alpha$:
$$ \gamma(x)(\alpha) \: \geq \: x(\q{ \gamma(x) \till \alpha  }).$$
\end{Def}


\begin{Lem}\label{second} There exists a pre-continuous and limit-closure preserving function $k$ such that $\gamma = \lim(k)$ is  strongly-converting.    \end{Lem}

\begin{proof} Fix a function $R: \klk \times \kappa \to \kappa$ which is injective and $R(\rho,\alpha) \geq \alpha$ for all $\rho$ and all $\alpha$.  

\p Define $k: \klk \to \klk$ inductively as follows:

\begin{itemize} 
\item $k(\sigma \cc \left<\beta\right>) := \begin{cases} k(\sigma) \cc  \left< \; R(\sigma \cc \left< \beta \right>, \beta) \; \right> & \text{ if } \len(\sigma)  = \q{k(\sigma)} \\ k(\sigma) & \text{ otherwise } \end{cases}$
\item 
For $\sigma$ of limit length (and  $\sigma=\varnothing$), $k(\sigma) :=  \bigcup \{k(\sigma') : \sigma' \subset \sigma\}  $.

\end{itemize}
We claim that $\gamma = \lim(k)$ is as required. Checking that $k$ is pre-continuous is easy. Let us check that $\gamma$ is strongly-converting. By construction,  for every $\alpha$, $\gamma(x)(\alpha) = \beta'$ iff there is some $\sigma \cc \left<\beta\right> \subseteq x$ such that

\begin{enumerate}
\item $\beta' = R(\sigma \cc \left< \beta \right>,\beta)$
\item $k(\sigma) = \gamma(x) \till \alpha$
\item $\len(\sigma) = \q{k(\sigma)}$ \end{enumerate}

 Therefore $\gamma(x)(\alpha) = \beta' \geq \beta  = x(\len(\sigma)) = x(\q{k(\sigma)}) = x(\q{\gamma(x) \till \alpha})$.

\p It remains to prove that $k$ is limit-closure-preserving. Since this is  very similar to the proof of Lemma \ref{higher}, we will leave out some details. Let $T$ be a limit-closed tree,  $U := \tr(k \pf T)$, and $\{u_i : i<\lambda\}$ an increasing sequence in $U$.  For each $i$, let $s_i \in T$ be minimal such that $u_i \subseteq k(s_i)$ (in this case, we actually have $u_i = k(s_i)$, but this is not relevant). 
As before, we will be done if we  prove the following claim:

\p \textbf{Claim.} $s_i \subseteq s_{i'}$ for all $i<i'$.

\begin{proof}  Suppose $s_i \not\subseteq s_{i'}$. Since $s_{i'} \subset s_i$ is impossible, we must have $s_i \bot s_{i'}$, so let $r$ be maximal with $r \subseteq s_i$ and $r \subseteq s_{i'}$. Again we must have $k(r) \subset u_i \subseteq u_{i'}$, hence we can find least $r_0$  with $r \subseteq r_0 \subseteq s_i$ and $k(r) \subset k(r_0)$, and least $r_1$ with $r \subseteq r_1 \subseteq s_{i'}$ and $k(r) \subset k(r_1)$. Moreover $r_0$ and $r_1$ are both of successor length, say with last digit $\beta_0$ and $\beta_1$, respectively. Then $k(r_0) = k(r) \cc \left< R(r_0, \beta_0)\right>$ and $k(r_1) = k(r) \cc \left< R(r_1, \beta_1)\right>$. Since $r_0 \neq r_1$ and $R$ is injective, we obtain a contradiction as before.  \qedhere (Claim) \end{proof} \vskip-0.2cm \end{proof}

It is clear that if $\gamma$ is strongly converting and $T \Vdash $ ``$\d$ is dominating'', then $T \Vdash $ ``$\gamma(\d)$ is strongly dominating''.  With this, we are ready to prove the final result.

\begin{Thm}[Main Theorem 3] \label{MainThm3} Assume $\klk =\kappa$. Suppose $\IP$ is a ${<}\kappa$-distributive tree forcing,  $h$ a pre-continuous function in the ground model with $f = \lim(h)$, and assume that  $T \Vdash$ ``$f(\xxg)$ is a dominating $\kappa$-real''. Then $T \Vdash$``there is a Cohen $\kappa$-real''. \end{Thm}

\begin{proof}  First we apply Lemma \ref{higher} to obtain a pre-continuous and limit-closure-preserving function $j$. Then, for $g = \lim(j)$,  it follows that $T \Vdash  $ ``$g(\xxg)$ is a dominating $\kappa$-real''. 

\p Now let $k$ and $\gamma$ be as in  Lemma \ref{second}. Then $T \Vdash   $ ``$\gamma(g(\xxg))$ is strongly dominating''.  

\p Let $\dot{e}$ be the name such that $T \Vdash \gamma(g(\xxg)) = \dot{e}$.  Since $k$ and $j$ are limit-closure-preserving, so is $k \circ j$. Therefore, by Lemma \ref{lemma3}, $\T_{\dot{e},T} = \tr ((k \circ j)''T)$ is limit-closed. Of course, the same applies for any stronger condition $S \leq T$, i.e.,  $\T_{\dot{e}, S}$ is also limit-closed for every $S \leq T$. This is all we need to apply  Lemma \ref{lemma1}, from which it follows that $T \Vdash $``there is a Cohen $\kappa$-real''. \end{proof} 

Unfortunately, none of the methods in this section seem to settle Question \ref{qqq}, which the authors consider very significant in the context of forcing over $\kk$: \emph{``Is it true that every ${<}\kappa$-distributive forcing adding a dominating $\kappa$-real adds a Cohen $\kappa$-real? Is it at least true for every ${<}\kappa$-closed forcing?''}

\s{Acknowledgments.} We would like to thank Hugh Woodin and Martin Goldstern for useful discussion and advice.

\bibliographystyle{plain}
\bibliography{Khomskii_Master_Bibliography}{}

\end{document}

%% file: newcommands.tex
\newcommand{\p}{\medskip \noindent}

\newcommand{\cc}{{^\frown}}

\newcommand{\till}{{\upharpoonright}}
\newcommand{\thru}{\uparrow}

\newcommand{\CH}{{\rm CH}}
\newcommand{\GCH}{{\rm \sf GCH}}

\newcommand{\PI}{\boldsymbol{\Pi}}

\newcommand{\ww}{\omega^\omega}
\newcommand{\wlw}{\omega^{<\omega}}

\newcommand{\stem}{{\rm stem}}
\newcommand{\Succ}{{\rm Succ}}

\newcommand{\IP}{\mathbb{P}}

\newcommand{\s}[1]{\medskip \noindent \textbf{#1}}

\newcommand{\M}{\mathcal{M}}

\newcommand{\IL}{\mathbb{L}}

\newcommand{\cov}{{\rm cov}}

\newcommand{\bb}{\mathfrak{b}}

\newcommand{\kk}{\mathfrak{k}}

\newtheorem{Thm}{Theorem}[section]

\newtheorem{Lem}[Thm]{Lemma}
\newtheorem{Cor}[Thm]{Corollary}

\newtheorem{Question}[Thm]{Question}

\theoremstyle{definition}
\newtheorem{Def}[Thm]{Definition}

\newtheorem{Remark}[Thm]{Remark}

\newcommand{\T}{{\mathcal{T}}}

